\documentclass{amsart}

\usepackage{amsthm} 
\usepackage{color}

\usepackage{graphicx,epsfig}
\usepackage{epstopdf}
\usepackage{amsmath, amssymb, latexsym, euscript,amscd}
\usepackage{url}
\usepackage[all]{xy}
\usepackage{psfrag}
\usepackage{mathrsfs}

\newcounter{notes}%

\definecolor{darkgreen}{rgb}{0.0, 0.5, 0.0}

\newtheorem{theorem}{Theorem}[section]
\newtheorem{lemma}[theorem]{Lemma}
 
\newtheorem{definition}[theorem]{Definition} 
\newtheorem{proposition}[theorem]{Proposition}

\def\Rep{\operatorname{Rep}}

\def\gap{\vspace{.3cm}\noindent}

\def\smallskip{\vspace{.15cm}}
\def\medskip{\vspace{.3cm}}
\def\text{\mbox}
\def\rh2{{\mathbb R}{\mathbb H}^2}
\def\ch2{{\mathbb C}{\mathbb H}^2}
\def\RR{{\mathbb R}}

\def\O{\operatorname{O}}
\def\ZZ{{\mathbb Z}}

\def\SS{{\mathbb S}}

\def\Fr{\operatorname{Fr}}

\def\interior{\operatorname{int}}
\def\SL{\operatorname{SL}}

\def\PGL{\operatorname{PGL}}
\def\GL{\operatorname{GL}}

\def\Hom{\operatorname{Hom}}
\def\H2R{{\mathbb H}^2\times {\mathbb R}}

\def\cl{\operatorname{cl}}

\def\C2{\operatorname{C^2}}

\def\Ccal{\mathcal C}
\def\boxset{\mathcal B}

\def\Pcal{\mathcal P}

\def\Dcal{\mathcal D}

\def\Xcal{\mathcal X}

\def\boxset{\mathfrak B}

\def\PO{\operatorname{PO}}
\def\bdy{\partial}
\def\RP{\operatorname{\mathbb RP}}
\def\vol{\operatorname{vol}}

\def\cm{\hat{\mu}}
\def\Diag{\operatorname{Diag}}

\definecolor{back}{RGB}{255,255,255}
\definecolor{fore}{RGB}{0,0,0}
\definecolor{title}{RGB}{255,0,90}

\definecolor{green}{rgb}{0.0, 0.5, 0.0}
\definecolor{purple}{rgb}{0.5, 0.0, 0.5}
\definecolor{bluegreen}{rgb}{0.0,0.5, 0.5}
\definecolor{orange}{rgb}{1,0.5, 0.1}
\definecolor{redgreen}{rgb}{0.5, 0.5, 0.0}

\def\green{\color{green}}

\def\dvol{\operatorname{dvol}}
\def\VFG{\operatorname{VFG}}

\def\green{\color{green}}

\def\g2{{\green 2}}

\newcommand{\bv}{\left[\begin{array}{c}}
\newcommand{\ev}{\end{array}\right]}
\newcommand{\bbmat}{\begin{bmatrix}} 
\newcommand{\ebmat}{\end{bmatrix}}
\newcommand{\bmat}{\begin{matrix}} 
\newcommand{\emat}{\end{matrix}}
\newcommand{\bpmat}{\begin{pmatrix}} 
\newcommand{\epmat}{\end{pmatrix}}

\begin{document}
\title{The space of properly-convex structures}
\date{\today}

\noindent\address{DC: Department of Mathematics, University of California, Santa Barbara, CA 93106, UlSA}\\
\address{ST: School of Mathematics and Statistics, The University of Sydney, NSW 2006, Australia}
\address{}\\
\email{cooper@math.ucsb.edu}\\
\email{stephan.tillmann@sydney.edu.au}

\author{Daryl Cooper  and Stephan Tillmann}
\begin{abstract} Suppose $G$ is finitely generated  group
and $\Ccal(G)$ consists
of all $\rho:G\to\PGL(n+1,\RR)$ for which there exists a properly convex
set in $\RP^n$ that is preserved by $\rho(G)$.
Then the
 image of $\Ccal(G)$ is closed in the character variety.

Suppose $G$ does not contain an infinite, normal, abelian subgroup
 and $\Dcal(G)\subset\Ccal(G)$ is the subset of holonomies
 of properly-convex $n$-manifolds with fundamental group $G$. Then the
 image $\Dcal(G)$ is closed in the character variety.

If $M$ is the interior of a compact $n$-manifold and $G=\pi_1M$ is as above,
and either $M$ is closed, or $\pi_1M$  contains a subgroup of infinite index
isomorphic to $\ZZ^{n-1}$, then
 $\Dcal(G)$ is closed.
If, in addition, $M$
 is the interior of a compact manifold $N$ such that every component of $\bdy N$ is $\pi_1$-injective,
 and  finitely
 covered by a torus, 
 then every element of $\Dcal(G)$ is the holonomy of a properly-convex structure on $M$,
and $\Dcal(G)$ is a union of connected components of a semi-algebraic set.
\end{abstract}
\maketitle

If $M$ is an $n$-manifold let $\Rep(M)=\Hom(\pi_1M,\PGL(n+1,\RR))$. This paper concerns the question of when is the subset of $\Rep(M)$ consisting of holonomies of
properly-convex real-projective structures on $M$  a closed subset of $\Rep(M)$ ? This was proved for closed surfaces in \cite{CW2}, \cite{ChoiGoldDef}
and for closed hyperbolic 3-manifolds in \cite{Kim}, and for closed manifolds that have trivial virtual center in \cite{Ben5}.
Marseglia \cite{Marseglia} proved it for strictly convex manifolds of finite Busemann measure. This implies
the ends are projectively equivalent to cusps of hyperbolic manifolds, \cite{CLT1}. His proof
follows the strategy of Benoist.
See \cite{MR3888622} for a survey.

The most general statement concerning the closure of the set of holonomies of properly convex structures
requires allowing homotopy equivalent, rather than homeomorphic, quotients in view of the following.
There is a $3$-manifold
$M$ that is the interior of a compact manifold, and
a sequence of representations 
 in $\Rep(M)$ that are holonomies of hyperbolic structures on $M$ and that converge to the holonomy of
 a hyperbolic structure
on a manifold $N$ that is homotopy equivalent, but not homeomorphic, to $M$; see  \cite{MR1411128}. 

Let  $\Dcal(M)$ be the subset of $\Hom(\pi_1M,\PGL(n+1,\RR))$ consisting
of discrete, faithful, torsion-free representations that preserve at least one properly convex set. The quotient of the domain is a
properly convex manifold that is homotopy equivalent to $M$.

\begin{theorem}\label{closed3} 
Suppose $M$  is
a properly convex $n$-manifold and  $\pi_1M$ is finitely generated and contains no non-trivial
normal abelian subgroup.
Suppose that either $M$ is closed, or $\pi_1M$  contains a subgroup $G\cong\ZZ^{n-1}$ of infinite index.

Then $\Dcal(M)$ is closed in $\Hom(\pi_1M,\PGL(n+1,\RR))$.\end{theorem}

With \cite{CT2}(2.3), this gives a short new proof of Benoist's theorem,
and of Marseglia's extension.
The hypotheses also hold if $M$ is the interior of a compact  3-manifold, and $\pi_2M=0$, and $M$ is not finitely
 covered by a bundle with fiber $F$ for which $\dim F\in\{1,2\}$ and  $\chi(F)\ge 0$.

If $N$ is a manifold with boundary then an element of $\pi_1N$
is called {\em peripheral} if it is represented by a loop that is freely homotopic into $\bdy N$.

A subgroup  $\Gamma\subset\GL(n+1,\RR)$ is {\em virtual flag group or VFG} if there is a finite index subgroup 
that is conjugate into the upper-triangular group, see \cite{CLT2}.
Given $n>0$ it follows from \cite{CLT2}(6.11) that there is an integer $k=k(n)>0$ 
such that if $C$ is an $(n-1)$-manifold that is finitely covered by a torus, and $\rho:\pi_1C\to\GL(n+1,\RR)$, then 
$\rho(\pi_1C)$ is a VFG
 if and only if for all $g\in \pi_1C$ every eigenvalue of  $\rho(g^k)$ is real.
If $N$ is a compact manifold with boundary, and $M=\interior(N)$
let $\VFG(M)\subset\Rep(M)$ be the semi-algebraic set that 
consists of those $\rho$ such that for all peripheral $\alpha\in\pi_1N$
 all the eigenvalues of $\rho(\alpha^{k})$ are real. 
 
 In this paper, a {\em generalized cusp} is a properly convex manifold homeomorphic to $C\times [0,1)$
with strictly convex boundary and $C$ is finitely covered by a torus. These are classified in \cite{BCL1}
and their moduli space is studied in \cite{BCL2}.

\begin{theorem}\label{closed4} 
Suppose $M$ is the interior of a compact connected $n$--manifold  $N$ 
and   each component $A\subset \bdy N$ is $\pi_1$-injective and $\pi_1A$
contains $\ZZ^{n-1}$ as a subgroup of finite index.
Also suppose $\pi_1M$ contains no non-trivial
normal abelian subgroup. Let $\Rep_{ce}(M)\subset\Rep(M)$ be the set of holonomies of properly convex structures on $M$.

Then $\Rep_{ce}(M)$
 is a union of connected components of $\VFG(M)$. Moreover the ends of every properly convex structure on $M$ are generalized cusps.\end{theorem}

When $M$ is closed then $\VFG(M)=\Rep(M)$ and one recovers Benoist's theorem.
When $N$ is a $3$-manifold and all the boundary components of $N$ are tori, then $\VFG(M)$ consists
of all representations of $\pi_1M$ such that the eigenvalues of all peripheral elements are real.
The proofs of these Theorems is based on an estimate (\ref{boxestimate}) for projective transformations that preserve a properly convex set.

{\em Sketch Proof}.
We want to show that, given a sequence
of convex projective structures $\Omega_k/\rho_k(\pi_1M)$ on $M$, if $\rho_k$ converges,
then $\Omega_k$ does not degenerate. Here $\Omega_k\subset\RP^n$ is properly convex and $\rho_k:\pi_1M\rightarrow\PGL(n+1,\RR)$. Consider the case that  $\Omega_k$
are getting thin in some {\em extreme} directions. This suggests some
 hyperbolic in $\rho_k(\pi_1M)$  translates
along a very short axis that cuts across $\Omega_k$ somewhere in the middle. 
Then the attracting and repelling fixed points are very close, so
the matrix must be very large. We seek to make this precise.

Fix an affine
patch and assume that  these extreme directions are given by coordinate axes. In other words the centroid
is at the origin and the inertia tensor is diagonal. By Fritz John's theorem \cite{Fritz}
there is an ellipsoid $E$ with center at the origin and $E\subset\Omega_k \subset n\cdot E$. There is a diagonal matrix $D_k$ possibly with some very large, and some very small,
entries that sends $E$ to the  sphere of radius $\sqrt{n}$, so that $$\boxset=[-1,1]^n\subset D_k(\Omega_k)\subset K\cdot \boxset,\qquad K=n^{3/2}$$
If $D_k\rho_k D_k^{-1}$ stays bounded, one has succeeded. Thus one is lead to think about  large matrices
$A\in \GL(n+1,\RR)$ that preserve $D_k(\Omega_k)$, and in particular send $\boxset$ into $K\cdot\boxset $.
The {\em box estimate} bounds all the entries $A_{ij}$ of $A$ by a fixed multiple
of $A_{n+1,n+1}$. One observes that {\em this estimate is not changed by conjugation by a diagonal matrix.}

The issue then becomes one of reducing to this special case, where all the $\Omega_k$ are in the same
nice position in one chosen affine patch. One might choose an arbitrary sequence of projective maps
$P_k$ so that $P_k(\Omega_k)$ is in nice position, but then maybe $P_k\to\infty$ and we
destroy the initial property of having both nicely positioned domains {\em and} bounded representations, since
$P_k\rho_kP_k^{-1}$ might be unbounded. This can be overcome by choosing  $P_k\in \O(n+1)$.
This leads to the notion of the {\em spherical center} of $\Omega_k$ which is an analog on the sphere of the centroid in Euclidean space.

We use an orthogonal transformation to position a properly convex domain $\Omega$
into a nice position in a chosen affine patch, and in the process select a {\em central} point in $\Omega$. By contrast, Benz{\'e}cri's theorem
\cite{Benz} uses an arbitrary
projective transformation to find a nice position for $\Omega$ such that an arbitrary point in $\Omega$
is {\em central}.\qed

Choi has a series of papers \cite{CHOIGS, Choi1,choiends, MR3715441} concerning the space of properly convex structures on orbifolds under various hypotheses,
and there is some overlap with our results.

Frequent use is made of \cite{CT2}(0.2)  concerning the irreducibility of the holonomy. The technique in \cite{CT2}
is largely algebraic, but in this paper it is more geometric, so we decided to separate them.
The first author thanks SMRI and Sydney University
for hospitality and partial support during the completion of this paper.   Research of the second author is supported in part under the Australian Research Council's ARC Future Fellowship FT170100316.
We also thank Darren Long for conversations concerning
an early version of this result.

\section{Box Estimate}

Let $\boxset=\prod_{i=1}^n[-1,1]\subset\RR^n\subset\RP^n$ and if $K>0$ then $K\cdot \boxset=\prod_{i=1}^n[-K,K]$ is called a {\em box}.

\begin{lemma}[box estimate]\label{boxestimate} If $A=(A_{ij})\in \GL(n+1,{\mathbb R})$
 and $K\ge1$ and $[A](\boxset)\subset K.\boxset$ then $$|A_{ij}|\le 2K\cdot|A_{n+1,n+1}|$$
\end{lemma}
\begin{proof} Set $\alpha=A_{n+1,n+1}$. Using the standard basis we have that 
$$[x_1e_1+\cdots +x_ne_n+e_{n+1}]=[x_1:x_2:\cdots:x_n:1]=(x_1,\cdots,x_n)\in\boxset\quad \Leftrightarrow\quad  \max_i |x_i|\le 1$$ 
First consider the entries $A_{i,n+1}$ in the last column of $A$. Since $[e_{n+1}]=0\in \boxset$ we have 
$$[Ae_{n+1}]=[A_{1,n+1}e_1+A_{2,n+1}e_2+\cdots A_{n,n+1}e_n+\alpha e_{n+1}]\in K\cdot\boxset$$ 
It follows that $|A_{i,n+1}/\alpha|\le K$. This establishes the bound
when $j=n+1$.

Next consider the entries $A_{n+1,j}$ in the bottom row with $j\le n$. Observe that 
$$p=[t e_j+e_{n+1}]\in\boxset\quad\Leftrightarrow\quad |t|\le 1$$
Then $[A]p=[A(te_j+e_{n+1})]\in K\cdot \boxset$. This is in ${\mathbb R}^n$ so the $e_{n+1}$ component
is not zero. Hence $t A_{n+1,j}+\alpha\ne 0$ whenever $|t|\le 1$ and it follows that $|A_{n+1,j}| <|\alpha|$.
Since $K\ge 1$ the required bound follows when $i=n+1$ and $j\le n$.

The remaining entries have $1\le i,j\le n$. If $|t|\le 1$ then   $[p_1:\cdots:p_{n+1}]=[A(te_j+e_{n+1})]\in K\cdot \boxset$
and  it follows that
$$|t|\le1\quad\Rightarrow\qquad \left|\frac{p_i}{p_{n+1}}\right|=\left|\frac{A_{i,n+1}+t A_{i,j}}{\alpha + t A_{n+1,j}}\right| \le K$$
 For all $|t|\le1$ the denominator is not zero hence $|A_{n+1,j}|<|\alpha|$. It follows that
$$|\alpha + t A_{n+1,j}|\le 2|\alpha|$$
Thus 
$$|t|\le1\quad\Rightarrow\qquad\left|A_{i,n+1}+t A_{i,j}\right| \le 2K\cdot|\alpha|$$
We may choose the sign of $t=\pm1$ so that the $A_{i,n+1}$ and $t A_{i,j}$ have the same sign. Then
$$ \left| A_{i,j}\right|\le \left|A_{i,n+1}+t A_{i,j}\right| $$
Which gives the result
$ \left| A_{i,j}\right|\le  2|\alpha|\cdot K$ in this remaining case.
\end{proof}

For example,  $A\in \O(n,1)$ preserves  the unit ball $B\subset\RR^n\subset\RP^n$ and 
$\boxset\subset B\subset\sqrt{n}\boxset$
thus this estimate applies to $A$ with $K=\sqrt{n}$.

\gap
 If $\Omega\subset\RR^n$ has finite positive Lebesgue measure the {\em centroid} of $\Omega$ is the point 
$${\cm}(\Omega)=\left.\int_{\Omega} x\ \dvol\right/\int_{\Omega} \dvol$$
If $\cm(\Omega)=0$ then
$$Q_{\Omega}(y)=\int_K\left(\|x\|^2\|y\|^2-\langle x,y\rangle^2\right) \dvol_x$$  
 is a positive definite quadratic form on  $\RR^n$ called the {\em inertia tensor}.
 
\begin{lemma}\label{Klemma} For each dimension $n$ there is $K=K(n)>1$ such if $\Omega\subset{\mathbb R}^n$ is an open bounded convex
set with inertia tensor $Q_{\Omega}=x_1^2+\cdots+x_n^2$ and centroid at the origin then $K^{-1}\boxset\subset\Omega\subset K\cdot\boxset$ and if $A\in\GL(\Omega)$ then $|A_{ij}|\le K\cdot|A_{n+1,n+1}|$.
\end{lemma}
\begin{proof} The first conclusion is given by the theorem of Fritz John, \cite{Fritz} with $K=\sqrt{n}$, see also  \cite{ball}.
Let $D=\Diag(K,\cdots,K,1)\in\GL(n+1,\RR)$ then $\boxset\subset D(\Omega)\subset K^2\boxset$.
Set $A'=D\cdot A\cdot D^{-1}$ then $A'\in\GL(D(\Omega))$, thus $|A'_{ij}|\le 2K^2\cdot|A'_{n+1,n+1}|$ by  (\ref{boxestimate}).
Now $|A'_{n+1,n+1}|=|A_{n+1,n+1}|$ and $|A_{i,j}|\le K^2 |A'_{i,j}|$ thus
$|A_{ij}|\le 2K^4\cdot|A_{n+1,n+1}|$. The result now holds using the constant $2K^4$.
  \end{proof}

\section{Spherical Centers}

An  subset $\Omega\subset\RP^n$  is {\em convex}  if every pair of points in $\Omega$
is contained in a segment of a projective line in $\Omega$, and {\em properly convex}
if in addition $\Omega$ is open and $\cl\Omega$ contains no projective line.
The {\em frontier} of $\Omega$ is $\Fr\Omega=\cl\Omega\setminus\interior\Omega$.
Morever $\Omega$ is {\em strictly convex} if it is properly convex
and  $\Fr\Omega$ contains no line segment, and $\Omega$ is called $C^1$ if for each $p\in\Fr\Omega$
there is a unique projective hyperplane $H$ such that $H\cap\cl\Omega=p$. Finally, $\Omega$ 
is {\em round} if it is both strictly convex, and $C^1$.
It is convenient to work in  the double cover, $\SS^n$, of $\RP^n$ and apply these terms to a lift of $\Omega$
to $\SS^n$.

If $\Omega\subset {\mathbb R}^n$ is bounded and convex then the centroid $\cm(\Omega)$ is a distinguished point in $\Omega$.
We wish to do a something similar for subsets of $\SS^n$.
Let $\langle\cdot,\cdot\rangle$ be the standard inner product on $\RR^{n+1}$
and let $\SS^n$ be the unit sphere. The open hemisphere that is
the $\pi/2$ neighborhood of $y\in\SS^n$ is $U_y=\{x\in \SS^n: \langle x,y\rangle>0\}$.
Radial projection $\pi_y:U_y\longrightarrow {\operatorname{T_y}}\SS^n$ 
from the origin onto the tangent space
to $\SS^n$ at $y$ is given by $$\pi_y(x)=\frac{x-\langle x,y\rangle y}{\langle x,y\rangle}$$
We may identify $U_y$ with an affine patch in $\RP^n$. Then this choice of inner product
on $\RR^{n+1}$ gives an inner product on the affine patch by using the radial
projection to identify the affine patch with the subspace $y^{\perp}\subset\RR^{n+1}$.
The {\em dual} of $\Omega\subset \SS^n$  is $$\Omega^*=\{y\in\SS^n:\ \forall x\in\cl \Omega\ \ \langle x,y\rangle >0\}$$
This set is always convex, but it is empty unless $\cl \Omega$ is  contained in some open hemisphere. 

\begin{definition} A point $y\in \Omega^*$ is called
a {\em center of $\Omega$}  if 
 $\cm(\pi_y(\Omega))=\pi_y(y)$.
 \end{definition}

If $y$ is a center of a properly convex set $\Omega$ then $y\in\Omega\cap\Omega^*$.
If $A\in\O(n+1)$ then  $(A\Omega)^*=A(\Omega^*)$, and  $Ay$ is a center of $A\Omega$
if  $y$ is a center of $\Omega$. 
 
\begin{theorem} If $\Omega\subset\SS^n$ is properly-convex and has nonempty interior 
then it has a center $p\in\Omega$.
\end{theorem}
\begin{proof} 
 There is a continuous map
$$m:\Omega^*\longrightarrow\Omega$$
defined by $m(y)=\pi_y^{-1}\cm(\pi_y\Omega)$. The theorem asserts
this map has a fixed point. First we prove this under the additional assumption that $\Omega$ is round.
With this assumption, if $y\in\bdy\overline\Omega^*$ then, since $\Omega$ is round,
$M(y)=y^{\perp}\cap\overline\Omega$ is a single point in $\bdy\overline\Omega$. This defines a homeomorphism
 $$M:\bdy\overline\Omega^*\longrightarrow\bdy\overline\Omega$$
Observe that the line segment $[y,M(y)]$ has length $\pi/2$ and is orthogonal to $\bdy\overline\Omega$
at $p$. Thus $M^{-1}$ is a spherical version of  the Gauss normal map.

{\bf Claim}  The extension of $m$ given by $m|\bdy\overline\Omega^*=M$  is continuous.
 
\gap
 Assuming the claim, if $m$ has no fixed point, define $r:\overline\Omega^*\longrightarrow\partial\overline\Omega$
as follows. Given $y\in\overline\Omega^*$ extend the line $[y,m(y)]$ to $[y,r(y)]$ where $r(y)\in\partial\overline\Omega$. 
Then $r$ is continuous 
and $r|\bdy\overline\Omega^*=M$  is a homeomorphism. 
This is impossible, which
proves the theorem modulo the claim.

Proof of claim:   The spherical metric $\theta$ on $\SS^n$ is given by
$\cos\theta(x,z)=\langle x,z\rangle$. 
The Euclidean norm $\|\cdot\|$  on $\pi_wU_w$ is given by $\|\pi_wx\|=\tan\theta(w,x)$.
 If $y\in\bdy\overline\Omega^*$ then
 $M(y)\in  y^{\perp}$ and it follows that  $\theta(y,p)=\pi/2$ where $p=M(y)=m(y)$.
 If $w\in\Omega^*$ is close to $y$ we must show that $m(w)$ is close to $p$.

Here is an informal argument:  $\pi_y\Omega$ is unbounded in the direction given by $p$, but bounded in all other directions.
For $w$ close to $y$ then $\pi_w\Omega$ only extends a large distance in the directions close to those given
by $p$.
Thus $\hat\mu(\pi_w\Omega)$ is far out in the direction of $p$.

Refer to Figure (\ref{sphproj}).
 Let $\ell\subset\SS^n$ be the geodesic segment of length $\pi/2$ with endpoints $y$ and $p$,
 and let $q$ be a point in the interior of $\ell$. Let 
  $H \subset \SS^n$ be the great sphere that is orthogonal to $\ell$ at $q$.
  Then $H$
and separates $\Omega$ into two components with closures $A$ and $B$, labelled so  that  $p$ is in $B$.
We will show that as $w$ approaches $y$  the volume in $\RR^n$ of $\pi_wA$ remains bounded and the volume
  of $\pi_wB$ goes to infinity. Let $a_w=\hat\mu(\pi_w A)$ and $b_w=\hat\mu(\pi_w B)$. Since $\pi_w A$ and $\pi_w B$ are
  convex $a_w\in \pi_wA$ and $b_w\in \pi_wB$. Using $\Omega=A\cup B$ gives
  
 $$\cm(\pi_w\Omega)=\frac{\vol_n(\pi_wA)\cdot a_w+\vol_n(\pi_wB)\cdot b_w}{\vol_n(\pi_wA)+\vol_n(\pi_wB)}=b_w+\left(\frac{\vol_n(\pi_wA)}{\vol_n(\pi_wA)+\vol_n(\pi_wB)}\right)(a_w-b_w)$$
    It follows that  $\cm(\pi_w\Omega)$ is close to $b_w$, thus $m(w)\in B$ for $w$ close enough to $y$. This
  proves the claim. It remains to estimate the volumes of $\pi_wA$ and $\pi_wB$.

Now $p$ is the unique point in $\overline\Omega$ that maximizes $\theta(y,p)$ and $\theta(y,p)=\pi/2$.
Since $B$ is a neighborhood of $p$ in $\overline\Omega$
there is $\epsilon>0$ such that $\theta(y,x)<\pi/2-2\epsilon$ for all $x\in A$.
If $w$ is close enough to $y$ then $\theta(w,y)<\epsilon$, so $\theta(w,x)\le \theta(w,y)+\theta(y,x)<\pi/2-\epsilon$ for all $x\in A$. 
Set $\alpha=\tan(\pi/2-\epsilon)$ then $\|x\|\le\alpha$ for all $x\in\pi_w(A)$,
thus $$\vol_n(\pi_w A)\le\omega_n\alpha^n$$ where $\omega_n$ is the volume of the unit ball in $\RR^n$. 
This  bound holds for all $w$ close enough to $y$. 

 To obtain a lower bound for the volume of $B$ we use that $B$ contains the cone on $J=H\cap\overline\Omega$
 from $p$, so
$$\vol_n(\pi_wB)\ge\vol_n(\operatorname{cone}(\pi_wp,\pi_wJ))= n^{-1}\cdot\vol_{n-1}(\pi_wJ)\cdot d(\pi_wp,\pi_wH)$$
By compactness, there is $\epsilon>0$ such that $\vol_{n-1}(\pi_w J)>\epsilon$ for all $w$ close to $y$.
Now $\|p\|\to\infty$ as $w\to y$. The line from $y$ to $p$ is orthogonal in $\SS^n$ to $H$. Thus for
$w$ close to $y$ the line in $\RR^n$ from $\pi_ww=0$ to $\pi_w p$ is almost orthogonal to $\pi_w H$.
It follows that $d(\pi_wp,\pi_wH)\ge \|p\|-\alpha-1$ for $w$ close enough to $y$. Thus $d(\pi_wp,\pi_wH)\to\infty$ as $w\to y$, hence
$\vol(\pi_w B)\to\infty$.
This completes the proof when $\Omega$ is round.

\begin{figure}[ht]
 \begin{center}
	 \includegraphics[scale=0.3]{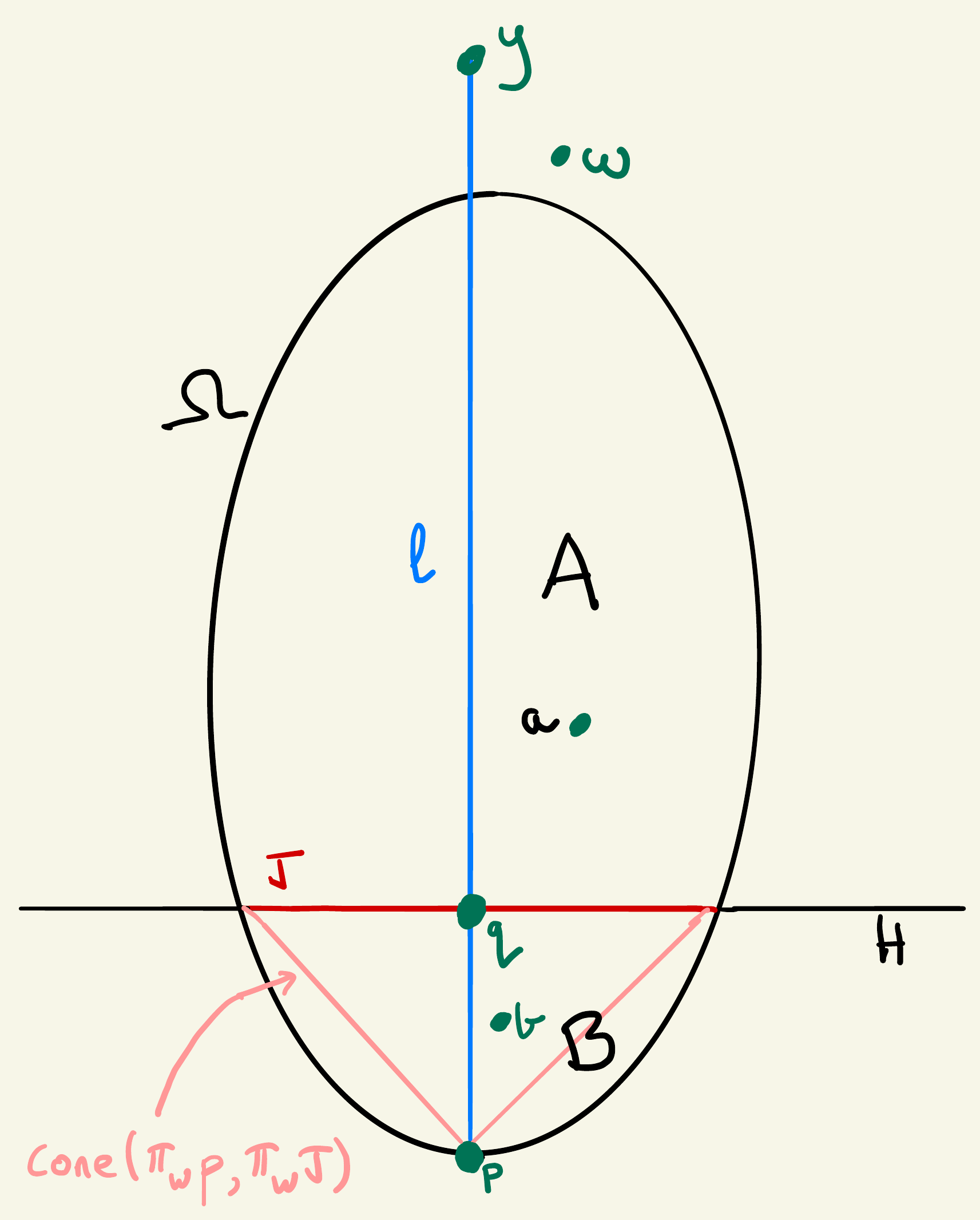}
	 \end{center}
 \caption{Image in $T_w\SS^n$}\label{sphproj}
\end{figure}

We now deduce the general result
from the fact that a general properly convex $\Omega$ is the limit of round convex sets.  
Given $\epsilon>0$, using the techniques in section 8 of \cite{CLT2},  there is a round $\Omega(\epsilon)$ such that
$$\Omega(\epsilon)\subset\Omega\subset N_{\epsilon}(\Omega(\epsilon))$$
 Let $c(\epsilon)\in\Omega(\epsilon)$\
be a center. There is a sequence $\epsilon_n\to 0$ such that $c_n\to c$. Clearly $c\in\overline\Omega$.
It is only necessary to check that the center of a round set
can not be very close to the boundary, then it follows that $c\in\Omega$. It is then easy to see that $c$ is a center for $\Omega$.
\end{proof}

\section{Limits of properly-convex manifolds}

Suppose $\rho_k$ is the holonomy of a properly-convex real projective structure on a manifold $M$ without boundary
of dimension $n$,
so that $M\cong \Omega_k/\Gamma_k$ with $\Omega_k$ properly-convex and $\Gamma_k=\rho_k(\pi_1M)$.
Suppose the holonomies converge
 $\lim\rho_k=\rho_{\infty}$. If $\Omega_{\infty}=\lim\Omega_k$ exists, then it is convex, but
it might have smaller dimension, or it might not be properly convex.  We describe this by saying {\em the domain
has degenerated}. If this happens then $\rho_{\infty}$ is reducible. The statements below allow $\bdy\Omega\ne\emptyset$
so that they can be applied to manifolds with $\bdy M\ne \emptyset$.

Let $\Pcal$ be the set of  properly convex open subsets of $\RP^n$. 
Choose
a Riemannian metric $d$ on $\RP^n$ then
 the {\em Hausdorff topology} on $\Pcal$ is given by the metric
$d(\Omega,\Omega')=\sup\{d(x,x'):x\in\Omega,\ \ x'\in\Omega'\}$.

 The {\em character variety} is
 $$\Xcal(G)= \Hom(G,\PGL(n+1,\RR))//\PGL(n+1,\RR)$$ Let
  $\chi: \Hom(G,\PGL(n+1,\RR))\to  \Xcal(\pi_1M)$ be projection. The following result makes no assumption 
 about the discreteness of representations.

\begin{proposition}[avoiding degeneration]\label{nocollapse} Suppose $G$ is a finitely generated group 
and $\rho_k:G\to\PGL(n+1,{\mathbb R})$ and $\rho_k(G)$ preserves a properly convex set  $\Omega_k\subset\RP^n$ with $\dim\Omega_k=n$.
Suppose $\lim\rho_k=\rho_{\infty}$. 

Then, after taking a subsequence, there are  $\beta_k\in\PGL(n+1,\RR)$ and $\sigma\in\Hom(G,\PGL(n+1,\RR))$
such that $\lim\beta_k\rho_k \beta_k^{-1}=\sigma$, and $\beta_k(\Omega_k)$ 
converges in the Hausdorff boundary topology to a properly-convex set $\Omega$ with $\dim\Omega=n$. In addition,
  $\sigma(G)$ preserves
 $\Omega$, and $\chi(\sigma)=\chi(\rho_{\infty})$. \end{proposition}
\begin{proof} Choose
an affine patch ${\mathbb R}^n\subset{\mathbb R}P^n$ then there is $\alpha_k\in \PO(n+1)$
such that $\Omega'_k=\alpha_k(\Omega_k)\subset{\mathbb R}^n$ has center $0\in {\mathbb R}^n$. We may choose $\alpha_k$ so that the
inertia tensor $Q_k=Q(\Omega'_k)$ is diagonal in the standard coordinates on ${\mathbb R}^n$. 
Since $\PO(n+1)$ is compact, after subsequencing
we may assume that $\alpha_k$ converges and also that $\rho'_k=\alpha_k\rho_k\alpha_k^{-1}$ converge. 
We now replace the original sequence $(\rho_k,\Omega_k)$ by the sequence $(\rho'_k,\Omega'_k)$.

Let $K=K(n)$ be given by (\ref{Klemma}). There is a unique positive diagonal matrix $D_k$
such that $Q_k=D_k^{-2}$. Set $\Omega'_k=D_k\Omega_k$, then $Q(\Omega'_k)=x_1^2+\cdots+x_n^2$.
By (\ref{Klemma}), there is $K>1$ depending only on $n$, such that $$K^{-1}\boxset\subset \Omega'_k\subset K\cdot\boxset$$
Given $g\in\pi_1M$ then $A=A(k,g)=\rho_k(g)\in \SL^{\pm}(n+1,{\mathbb R})$ preserves $\Omega_k$.
 The matrix $B=B(k,g)=D_k A(k,g) D_k^{-1}$ preserves $\Omega'_k$. By (\ref{Klemma})  $$\forall i,j\quad |B_{i,j}|\le K\cdot |B_{n+1,n+1}|$$ Since $D_k$ is diagonal it follows that $$B_{n+1,n+1}=A_{n+1,n+1}$$
The key to our approach is that this estimate is not affected by how large $D_k$ is.
Now $A_{n+1,n+1}=A(k,g)_{n+1,n+1}$ converges as $k\to\infty$ for each $g$. Hence the entries of $B(k,g)$ are uniformly bounded for fixed
$g$ as $k\to\infty$.
 Thus we may pass to a subsequence where $B(k,g)=D_k\rho_k(g)D_k^{-1}$ converges
for every $g\in\pi_1M$, and this gives a limiting representation $\sigma=\lim D_k\rho_kD_k^{-1}$. 

The subspace of $\Pcal$ consisting of properly convex $\Omega$
with $K^{-1}\cdot \boxset\subset\Omega\subset K\cdot\boxset$ is compact. Therefore there is a subsequence
 so that
$\Omega=\lim\Omega_k'$ exists. 
 Then $K^{-1}\boxset\subset \Omega'\subset K\cdot\boxset$
so $\Omega'$ is properly convex, and $\dim\Omega'=n$, and $\sigma$
preserves $\Omega'$.
Clearly $\chi(\beta_k\rho_k \beta_k^{-1})=\chi(\rho_k)$,
and taking limits gives
 $\chi(\sigma)=\chi(\rho_{\infty})$.
\end{proof}

Suppose $G$ is a finitely generated group. Let $\Ccal(G)\subset \Hom(G,\PGL(n+1,\RR))$ be the subset of all $\rho$
such that that there exists a properly convex $\Omega\subset\RP^n$ with $\dim\Omega=n$ and $\rho(G)$ preserves $\Omega$. {\em Observe that the trivial representation has this property.}
The following makes no assumption about faithful or discrete.

\begin{theorem}\label{closed1}  If $G$ is finitely generated then $\chi(\Ccal(G))$
 is a closed subset of $\Xcal(G)$.
\end{theorem}

Let $\Dcal(G)\subset\Ccal(G)$ be the subset of discrete faithful representations. 
The limit of discrete faithful representations is not always discrete and faithful
without an additional hypothesis on $G$.

\begin{theorem}\label{closed2} Suppose $G$ is a finitely generated group that does
 not contain a non-trivial normal abelian subgroup.
  Let $\Dcal(G)\subset\Hom(G,\PGL(n+1,{\mathbb R)})$ be the subset of holonomies
 of properly-convex $n$-manifolds with fundamental group $G$. Then the
 image $\Dcal(G)$ is closed in the character variety.

\end{theorem}
\begin{proof}
 With the setup of (\ref{nocollapse}) if $\rho_k\to\sigma$ and all the $\rho_k$ are discrete faithful, then  $\sigma$ is discrete and faithful by  \cite{CT2}(1.2).
\end{proof}

We would like to know that the set of discrete faithful
 representations are closed, rather than just the image of this set in
the character variety. 
However  the set of holonomies of  properly convex structures
on a given manifold $M$ is not always a closed subset of the representation variety. 
In general $\sigma$ and $\rho_{\infty}$ in (\ref{closed2}) are not conjugate. An additional group theoretic assumption on
$\pi_1M$ suffices to ensure that $\sigma$ and $\rho_{\infty}$ are irreducible, and therefore conjugate.
When $M$ is closed, Benoist showed that it suffices that $\pi_1M$ has virtually trivial center suffices. We extend this result.

\begin{proof}[Proof of (\ref{closed3})] With the notation of (\ref{closed2}) we may assume $\lim\rho_k=\rho_{\infty}$
and $\rho_{\infty}$ is conjugate to $\sigma$.
Since $\pi_1M$ contains no non-trivial normal abelian subgroup, and either $M$ is closed or contains
$\ZZ^{n-1}$ with infinite index, then $\sigma$ is irreducible by  \cite{CT2}(0.2). 
Now $\chi(\rho_{\infty})=\chi(\sigma)$, and an irreducible representation is determined
up to conjugacy by the character, from which it follows that $\rho_{\infty}$ is conjugate to $\sigma$.
Thus, after conjugating the original sequence, and taking a subsequence, $\lim \Omega_k=\Omega_{\infty}$
and $\lim\rho_k=\rho_{\infty}$. Then $\Omega_{\infty}/\rho_{\infty}(\pi_1M)$ is a manifold that is homotopy equivalent to $M$.
\end{proof}

We wish to show that the limit manifold $N=\Omega_{\infty}/\rho_{\infty}(\pi_1M)$ is homeomorphic, rather than just
 homotopy equivalent,  to $M$. If $M$ is closed, then this is guaranteed using Gromov-Hausdorff convergence.
 When $M$ is not compact then in general there are counterexamples.
However, if the ends of $M$ are {\em generalized cusps} then, since $N$ is 
homotopy equivalent to $M$,
 the ends of $N$ are also generalized cusp. This ensures $M$ and $N$ are homeomorphic, as is shown below.

\begin{proof}[Proof of Theorem \ref{closed4}]  We wish to apply (\ref{closed3}). If $M$ is not closed,
 then $N$ has a boundary component $A$ and $\pi_1A$ contains a subgroup $G\cong\ZZ^{n-1}$.
Then $|\pi_1M:G|=\infty$
 otherwise  the intersection of the conjugates of $G$ is a non-trivial
normal abelian subgroup of $\pi_1M$. Thus we may apply  (\ref{closed3}) with this $G$.

In what follows if $X$ and $Y$ are spaces then $X\cong Y$ means $X$ is homeomorphic to $Y$.
Suppose $\rho\in \cl(\Rep_{ce}(M))$. Now $\Rep_{ce}\subset\Dcal(\pi_1M)$ 
and   $\Dcal(\pi_1M)$ is closed by (\ref{closed3}) so $\rho\in\Dcal(\pi_1M)$. Thus
 $\Gamma=\rho(\pi_1M)$ preserves a properly convex 
domain $\Omega(\rho)$.
Let $N=\Omega(\rho)/\Gamma$ be the corresponding manifold.
Then by \cite{CT2}(0.1) the ends of $N$ are generalized cusps.
By \cite{CLT2}(0.2) for $\sigma\in\Rep_{ce}(M)$ close to $\rho$ there is a properly convex manifold 
$N_{\sigma}=\Omega(\sigma)/\sigma(\pi_1M)\cong N$.
Since $\sigma\in\Rep_{ce}(M)$ there is also $\Omega'$ with $M_{\sigma}=\Omega'/\sigma(\pi_1M)\cong M$. 
The minimal convex submanifolds of $M_{\sigma}$ and $N_{\sigma}$ are the same by \cite{CT2}(3.5)(9), so 
$M\cong M_{\sigma}\cong N_{\sigma}\cong N$.
Thus $\rho\in\Rep_{ce}(M)$.
\end{proof}

\small
\bibliography{Koszulrefs.bib} 
\bibliographystyle{abbrv} 

\end{document}